\documentclass[10pt,a4paper]{amsart}

\usepackage{amssymb}
\usepackage{amsthm}
\usepackage{amsfonts}
\usepackage{microtype}
\usepackage{mathrsfs}
\usepackage{graphicx}
\usepackage{color}
\usepackage{latexsym}
\usepackage{rotating}
\usepackage{xspace}
\usepackage[all]{xy}
\usepackage{longtable}
\usepackage{leftidx}
\usepackage{mathtools}
\usepackage{titlesec}
\usepackage{tikz}
\usepackage{url}
\usepackage{amscd}
\usepackage{tikz-cd}
\usetikzlibrary{calc}
\usetikzlibrary{arrows}
\usetikzlibrary{decorations.markings}
\usepackage{geometry}

\newtheorem{theorem}{Theorem}[section]
\numberwithin{equation}{theorem}
\newtheorem{lemma}[theorem]{Lemma}
\newtheorem{proposition}[theorem]{Proposition}
\newtheorem{corollary}[theorem]{Corollary}

\theoremstyle{definition}

\newtheorem{example}[theorem]{Example}

\theoremstyle{conjecture}

\newtheorem{question}[theorem]{Question}

\newcommand{\Ass}{\operatorname{Ass}}

\newcommand{\Min}{\operatorname{Min}}

\newcommand{\Spec}{\operatorname{Spec}}

\newcommand{\Ht}{\operatorname{ht}}

\newcommand{\Ann}{\operatorname{Ann}}

\newcommand{\Nil}{\operatorname{Nil}}

\newcommand{\Max}{\operatorname{Max}}

\newcommand{\lo}{\longrightarrow}

\newcommand{\fm}{\frak{m}}
\newcommand{\fp}{\frak{p}}
\newcommand{\fq}{\frak{q}}
\newcommand{\fr}{\frak{r}}

\newcommand{\fn}{\frak{n}}

\newcommand{\suchthat}{\;\ifnum\currentgrouptype=16 \middle\fi|\;}

\makeatletter
\newcommand{\holim@}[2]{%
\vtop{\m@th\ialign{##\cr
\hfil$#1\operator@font holim$\hfil\cr
\noalign{\nointerlineskip\kern1.5\ex@}#2\cr
\noalign{\nointerlineskip\kern-\ex@}\cr}}%
}
\newcommand{\holim}{%
\mathop{\mathpalette\holim@{\rightarrowfill@\textstyle}}\nmlimits@
}
\makeatother

\makeatletter
\def\@secnumfont{\bfseries}
\def\section{\@startsection{section}{1}%
\z@{.7\linespacing\@plus\linespacing}{.5\linespacing}%
{\normalfont\Large\bfseries\filcenter}}
\def\subsection{\@startsection{subsection}{2}%
\z@{.5\linespacing\@plus.7\linespacing}{-.5em}%
{\normalfont\large\bfseries}}
\makeatother
\begin{document}
\title[Characterization and examples of ...]
{Characterization and examples of commutative isoartinian rings}

\author[A. Daneshvar and K. Divaani-Aazar]
{Asghar Daneshvar and Kamran Divaani-Aazar}
\address{A. Daneshvar, Department of Mathematics, Faculty of Mathematical Sciences, Alzahra University, Tehran, Iran.}
\email{a.daneshvar@alzahra.ac.ir, a.daneshvar@ipm.ir}

\address{K. Divaani-Aazar, Department of Mathematics, Faculty of Mathematical Sciences, Alzahra
University, Tehran, Iran.}
\email{kdivaani@ipm.ir}

\begin{abstract}
	Noetherian rings have played a fundamental role in commutative algebra, algebraic number theory, and algebraic geometry. Along
	with their dual, Artinian rings, they have many generalizations, including the notions of isonoetherian and isoartinian rings.
	In this paper, we prove that the Krull dimension of every isoartinian ring is at most one. We then use this result to provide
	a characterization of isoartinian rings. Specifically, we prove that a ring $R$ is isoartinian if and only if $R$ is uniquely
	isomorphic to the direct product of a finite number of rings of the following types:
	(i) Artinian local rings;
	(ii) non-Noetherian isoartinian local rings with a nilpotent maximal ideal;
	(iii) non-field principal ideal domains;
	(iv) Noetherian isoartinian rings $A$ with $\Min A$ being a singleton and $\Min A \subsetneq \Ass A$;
	(v) non-Noetherian isoartinian rings $A$ with $\Min A$ being a singleton and $\Min A \subsetneq \Ass A$;
	(vi) non-Noetherian isoartinian rings $A$ with a unique element in $\Min A$ that is not maximal, and $\Min A=\Ass A$.
	Several examples of these types of rings are also provided.
	
	\vspace{0.4cm}

	

\end{abstract}

\maketitle

\section{Introduction}

Throughout this article, the term ``ring" refers to commutative rings with nonzero identity.

In 1921, Emmy Noether showed out that in a ring $R$, every ideal of $R$ is finitely generated if and only if $R$ satisfies ACC on its ideals. Noether’s subsequent deep research had a significant impact
on module theory and representation theory. Assuming ACC on a ring leads to many beautiful properties. For instance, Emanuel Lasker and Noether demonstrated that every ideal of a ring satisfying ACC has a
primary decomposition, which is a generalization of the fundamental theorem of arithmetic. In 1927, Emil Artin began studying rings with both ACC and DCC properties, which led to several interesting applications of ACC property. It was clear from simple examples that ACC does not imply DCC, so Artin assumed both ACC and DCC. In honor of Noether and Artin, the rings satisfying ACC (respectively, DCC) are now referred to as Noetherian (respectively, Artinian) rings.

The method of decomposition is a common approach in mathematical research, which enables the transfer of properties from basic
building blocks to more complex objects and vice versa. This technique has been used for a long time, with the decomposition of
natural numbers into prime factors being one of the earliest examples. Algebra is no exception, and the fundamental theorem of
finitely generated Abelian groups is just one of many instances of this approach. Another famous example is the theorem attributed
to Akizuki-Cohen, which states that an Artinian ring is uniquely isomorphic to the direct product of a finite number of Artinian
local rings.

Due to the importance and wide range of applications of Noetherian and Artinian rings, many authors have proposed and examined some
generalizations of these concepts. One of the most natural generalizations was introduced by Alberto Facchini and Zahra Nazemian in
2016 \cite{FN3}. They define a ring $R$ to be isonoetherian (respectively, isoartinian) if every ascending (respectively, descending)
chain of ideals of $R$ terminates up to isomorphism. Several interesting results and examples pertaining to these types of rings can
be found in \cite{FN3}. To explore the topic of isonoetherian and isoartinian rings in more depth, we recommend consulting the following references: \cite{FN1, FN2, FN3, BDV1, BDV2}.

Numerous natural questions arise regarding isonoetherian and isoartinian rings. For instance, it is unknown whether an isoartinian
ring is necessarily isonoetherian. In \cite[Corollary 4.8]{FN3}, Facchini and Nazemian showed that a reduced isoartinian ring is
isomorphic to a finite direct product of principal ideal domains. In particular, every reduced isoartinian ring is Noetherian. The
aim of this paper is to establish a decomposition result for general isoartinian rings. We prove that:

\begin{theorem}\label{1.1}  A ring $R$ is isoartinian if and only if it is uniquely isomorphic to the direct product of a finite
number of the following types of rings with various choices:
\begin{enumerate}
\item[(i)] Artinian local rings.
\item[(ii)] non-Noetherian isoartinian local rings with nilpotent maximal ideal.
\item[(iii)] non-field principal ideal domains.
\item[(iv)] Noetherian isoartinian rings $A$ with a unique minimal prime ideal $\fp$ such that $A/\fp$ is a principal
ideal domain and $\Min A\varsubsetneq \Ass A$.
\item[(v)] non-Noetherian isoartinian rings $A$ with a unique minimal prime ideal $\fp$ such that $A/\fp$ is a principal ideal
domain and $\Min A\varsubsetneq \Ass A$.
\item[(vi)] non-Noetherian isoartinian rings $A$ with a unique minimal non-maximal prime ideal $\fp$ such that $A/\fp$ is a principal
ideal domain and $\Min A=\Ass A$.
\end{enumerate}
\end{theorem}
In the proof of this theorem, a key ingredient is to demonstrate that for an isoartinian ring $R$, the set of associated primes of $R$
is finite, and every minimal prime ideal of $R$ is also an associated prime of $R$. This implies, in particular, that the Krull dimension
of every isoartinian ring is at most one.

While there are many examples of rings falling under types (i) and (iii) in the above theorem, the other four types of isoartinian rings
deserve more attention. We provide some examples of these types. One of these examples demonstrates that isoartinian rings may not be
isonoetherian. Furthermore, we provide some examples to show that the condition of being ``isoartinian" cannot be relaxed for types (ii),
(v) and (vi).

\section{Main results}

In the proof of our main result, we will utilize the fact that the Krull dimension of every isoartinian ring is at most one, as stated in Proposition \ref{2.3c}. To demonstrate this, we will rely on the following two lemmas.

Recall that the set of associated primes of $R$ is defined as $$\Ass R=\{\fp\in \Spec R \mid \fp=(0:_Rx) \
\text{for\ some}\ x\in R\}.$$ We use $Z(R)$ to represent the set of zero-divisors of $R$.
 An ideal $I$
of a ring $R$ is said to be {\it regular} if it contains a regular element of $R$. A ring $R$ is called {\it Marot} if every regular ideal
of $R$ can be generated by regular elements. It is well-known that if $Z(R)$ is a finite union of prime ideals, then $R$ is a Marot ring.

\begin{lemma}\label{2.1a} Let $R$ be an isoartinian ring. Then $\Ass R$ is nonempty and finite. In particular, $R$ is a Marot ring.
\end{lemma}

\begin{proof} Since $R$ is isoartinian, it satisfies both ACC and DCC on annihilator ideals, as shown in \cite[Lemma 4.10(1)]{FN3}. 
Let $\Sigma:=\{(0:_Rx) \mid 0\neq x\in R\}$, which is nonempty since $R$ is nonzero. As $R$ has ACC on annihilator ideals, $\Sigma$ 
has a maximal element $(0:_Rz)$ with respect to inclusion. One can easily observe that the ideal $(0:_Rz)$ is prime, and hence 
$(0:_Rz)$ belongs to $\Ass R$.

Now, we show that $\Ass R$ is finite. Suppose to the contrary that $\Ass R$ is infinite. Then there are prime ideals $\fp_1, \fp_2,
\ldots, \fp_n, \ldots \in \Ass R$ such that $\fp_1$ is a maximal element of $\Ass R$ and $\fp_{n+1}$ is a maximal element of $\Ass R
\setminus \{\fp_1, \fp_2,\ldots, \fp_n\}$ for all $n\geq 1$. Consider the descending chain of ideals $$\fp_1\supseteq \fp_1\cap \fp_2
\supseteq \fp_1\cap \fp_2\cap \ldots \cap \fp_n\supseteq \cdots .$$ Since $R$ has DCC on annihilator ideals, this chain stabilizes.
Thus, there is a natural number $\ell$ such that $\fp_1\cap \fp_2 \cap \ldots \cap \fp_{\ell}=\fp_1\cap \fp_2 \cap \ldots\cap
\fp_{\ell}\cap \fp_{\ell+1}$. Thus, $\fp_i\subseteq \fp_{\ell+1}$ for some $1\leq i \leq \ell$.  However, $\fp_i$ is a maximal element
of $\Ass R\setminus \{\fp_1, \fp_2,\ldots, \fp_{i-1}\}$, which implies that $\fp_i=\fp_{\ell+1}$. This is a contradiction, since
$\fp_{\ell+1}\in  \Ass R\setminus \{\fp_1, \fp_2,\ldots, \fp_{i}, \ldots, \fp_{\ell}\}$.

Finally, note that $Z(R)$ is the union of members of $\Ass R$, and that if $Z(R)$ is a finite union of prime ideals, then $R$ is a
Marot ring. Therefore, $R$ is a Marot ring.
\end{proof}

We use $\Spec R$ to denote the set of all prime ideals of $R$. Also, let $\Max R$ (respectively  $\Min R$) denote the set of all maximal
(respectively minimal) prime ideals of $R$.

\begin{lemma}\label{2.2b} Let $R$ be an isoartinian ring. Then $\Min R$ is a nonempty subset of $\Ass R$. As a consequence, $\Min R$ is
finite.
\end{lemma}

\begin{proof} First, note that since $R$ is nonzero, $\Min R$ is nonempty. It follows from the definition that the localization of an
isoartinian ring at any multiplicative set is also isoartinian, as shown in \cite[Lemma 3.2]{Da}. Let $\fp\in \Min R$. Then $\fp R_{\fp}$
is the unique prime ideal of the isoartinian ring $R_{\fp}$. According to Lemma \ref{2.1a}, we have $\Ass R_{\fp}=\{\fp R_{\fp}\}$. Hence,
there exists $x/t\in R_{\fp}$ such that $\fp R_{\fp}=(0:_{R_{\fp}}x/t)$. Clearly, $(0:_{R_{\fp}}x/t)=(0:_{R_{\fp}}x/1)$.

Let $\Sigma$ denote the nonempty set of all annihilator ideals of the form $(0:_RI)$, where $I$ is a finitely generated ideal of $R$ and $I\subseteq \fp$. Since $R$ has DCC on annihilator ideals, $\Sigma$ admits a minimal element $(0:_RJ)$. We claim that $(0:_RJ)=(0:_R\fp)$.
Suppose the contrary holds. Then as $(0:_R\fp)\subseteq (0:_RJ)$, there exists an element $y\in (0:_RJ)\setminus (0:_R\fp)$. Since $y\notin (0:_R\fp)$, there exists $a\in \fp$ such that $ay\neq 0$. As $(0:_RJ+Ra)\subseteq (0:_RJ)$, by the minimality of $(0:_RJ)$ in $\Sigma$, we
deduce that $(0:_RJ+Ra)=(0:_RJ)$. Thus, $y\in (0:_RJ+Ra)$, which implies $ay=0$, a contradiction.

As $x/1\in (0:_{R_{\fp}}\fp R_{\fp})\subseteq (0:_{R_{\fp}}JR_{\fp})$ and $J$ is finitely generated, there exists $s\in R\setminus \fp$
such that $J(sx)=0$. Consequently, $sx\in (0:_RJ)=(0:_R\fp)$, and hence $\fp\subseteq (0:_Rsx)$. On the other hand, we have
$$(0:_Rx)\subseteq (0:_{R_{\fp}}x/1)\cap R=\fp R_{\fp}\cap R=\fp,$$ so $(0:_Rsx)\subseteq \fp$. Therefore, $\fp=(0:_Rsx)$, and so
$\fp\in \Ass R$.
\end{proof}

We can now prove the claim we made at the beginning of this section:

\begin{proposition}\label{2.3c} Let $R$ be an isoartinian ring. Then $R/\fp$ is a principal ideal domain for every $\fp\in \Ass R$.
Consequently $\dim R\leq 1$, and so $\Spec R=\Min R\cup \Max R$.
\end{proposition}

\begin{proof}  Let $\fp \in \Ass R$. Then, $\fp=(0:_Rx)$ for some nonzero element $x$ of $R$. As $Rx$ is a submodule of $R$ and
$R/\fp\cong Rx$, the ring $R/\fp$ is isoartinian. By \cite[Corollary 4.8]{FN3}, we deduce that $R/\fp$ is a principal ideal domain.

Clearly, $\dim R=\dim R/\fq$ for some $\fq\in \Min R$. By Lemma \ref{2.2b}, we have $\fq\in \Ass R$. As $R/\fq$ is a principal
ideal domain, it follows that $\dim R=\dim R/\fq \leq 1$, as desired.
\end{proof}

To prove the main result, we need the following six additional lemmas. We begin by showing that isoartinian rings are a subclass of a
well-known class of rings called Pr\"ufer rings.

Recall that the localization of a ring $R$ at the set of all regular elements is called the {\it total ring of fractions} of $R$ and
is denoted by $T(R)$. An ideal $I$ of a ring $R$ is said to be {\it invertible} if $I^{-1}I=R$, where $$I^{-1}=\{r\in T(R) \mid rI
\subseteq R\}.$$
A ring $R$ for which every finitely generated regular ideal is invertible is called a {\it Pr\"ufer ring}. Following \cite{Gri},
if every regular ideal of a ring $R$ is invertible, then we call it a {\it Dedekind ring}. So, Dedekind rings are Pr\"ufer.
 

\begin{lemma}\label{2.4d} Let $R$ be an isoartinian ring. Every regular ideal of $R$ is principal and generated by a regular element. Consequently, $R$ is a Dedekind ring.
\end{lemma}

\begin{proof} Let $I$ be a regular ideal of $R$. Then, $I$ contains an element $z$ such that $(0:_Rz)=0$. Thus, $I\cong  Iz^k$ for
every natural number $k$. Consider the following descending chain of ideals of $R$ $$R\supseteq I\supseteq Rz\supseteq Iz\supseteq
Rz^2\supseteq Iz^2 \supseteq Rz^3\supseteq \cdots.$$ As $R$ is isoartinian, there is a natural number $k$ such that $I\cong Iz^k
\cong Rz^{k+1}$. It follows that $I$ is principal. Clearly, $z^{k+1}$ is regular. Let $\varphi: Rz^{k+1}\lo I$ be the mentioned
isomorphism. Then, $I=R\varphi (z^{k+1})$ and $\varphi (z^{k+1})$ is regular.
	
For the last statement, let $J$ be a finitely generated regular ideal. By the first statement, $J=Rz_0$ for some regular
element $z_0$. As $1=(\frac{1}{z_0})z_0$, it follows that $R\subseteq J^{-1}J$. Therefore, $J^{-1}J=R$, and so $R$ is a Dedekind
ring. Note that $J^{-1}J\subseteq R$ is evident.
\end{proof}	

A {\it discrete valuation ring} (DVR) is a principal ideal domain with a unique nonzero maximal ideal. An Artinian local principal
ideal ring is called {\it special}, and it has only finitely many ideals, each of which is a power of the maximal ideal.

\begin{lemma}\label{2.5e} Assume that $R$ is a local ring with a principal maximal ideal ${\fm}$. Then the following are equivalent:
\begin{enumerate}
\item[(i)] $R$ is Noetherian.
\item[(ii)]  $\bigcap_{i=1}^\infty{\fm}^i=0$.
\item[(iii)] $\{0\}\cup \{\fm^i \mid i\in \mathbb{N}_0\}$ is the set of all ideals of $R$, and so $R$ is either a DVR or
a special ring.
\end{enumerate}
\end{lemma}

\begin{proof} (i)$\Rightarrow$(ii) is immediate by the Krull intersection theorem.
	
(ii)$\Rightarrow$ (iii) Let ${\fm}=Rz$ and $I$ be a nonzero ideal of $R$. Then there exists a natural number $\ell$ such that $I\subseteq \fm^{\ell}$ and $I\nsubseteq \fm^{\ell+1}$. Choose $x\in I\setminus \fm^{\ell+1}$. Since $x\in \fm^{\ell}\setminus \fm^{\ell+1}$, we have $x=rz^{\ell}$ for some $r\in R\setminus \fm$. Since $r$ is a unit, we have $z^{\ell}\in I$, and so $\fm^{\ell}=Rz^{\ell}=I$. Thus,
$\{0\}\cup \{\fm^i \mid i\in \mathbb{N}_0\}$ is the set of all ideals of $R$. In particular, $R$ is a principal ideal ring. If $\fm$ is
nilpotent, then $R$ is Artinian, and so $R$ is a special ring.

Now, assume that $\fm$ is not nilpotent, so $z^k\neq 0$ for all $k\in \Bbb N$. We will show that $R$ is a DVR. Since $R$ is a local principal ideal ring, and $\fm\neq 0$, it suffices to show that $R$ is a domain. Let $x$ be a nonzero element of $R$. Since $\bigcap_{i=1}^\infty{\fm}^i
=0$, there exists a natural number $t$ such that $x\in {\fm}^t\setminus {\fm}^{t+1}$. Hence, $x=rz^t$ for some unit $r$. Thus, every nonzero element of $R$ can be written as $uz^n$, where $u$ is a unit of $R$ and $n$ is a natural number. Now, a direct examination shows that $R$ is
a domain. Therefore, $R$ is a DVR.
	
(iii)$\Rightarrow$(i) is evident.
\end{proof}

\begin{lemma}\label{2.6f} Let $R$ be an isoartinian ring. Then $R$ is isomorphic to the direct product of finitely many indecomposable
isoartinian rings. Moreover, if $$R\cong R_1\times \cdots \times R_l\cong S_1\times \cdots \times S_m$$ where $R_i$'s and $S_j$'s are indecomposable isoartinian rings, then $l=m$ and $R_i\cong S_{\sigma (i)}$ for some permutation $\sigma$ on the set $\{1,\dots, l\}$.
\end{lemma}	

\begin{proof} First, note that the (minimal) prime ideals of $\prod_{i=1}^l R_i$ are of the form $\prod_{i=1}^l{\fp}_i$, where for some
$j$, $\fp_j$ is a (minimal) prime ideal of the ring $R_j$ and ${\fp}_i=R_i$ for all $i\neq j$.

By Lemma \ref{2.2b}, $\Min R$ is finite. Let $n$ be the number of elements of $\Min R$. If $R$ is indecomposable, then we are done.
Otherwise, suppose $R\cong R_1\times R_2$, where $R_1$ and $R_2$ are two rings. If $R_1$ and $R_2$ are indecomposable, then we are
done. Without loss of generality, assume that $R_1$ is not indecomposable, so $R_1\cong A_1\times A_2$ for some rings $A_1$ and $A_2$.
Thus, we have $R\cong A_1\times A_2\times R_2$. If $A_1$, $A_2$ and $R_2$ are all indecomposable, then we are done. Otherwise, we
repeat the process. Keeping in mind the form of prime ideals of the direct products of rings, this process must stop after at most $n$ steps.
All of the finite decomposition components are isoartinian by
\cite[Lemma 4.1]{FN3}.
	
For the second assertion, let $\phi: R_1\times \cdots \times R_l\lo R$ and $\psi:  S_1\times \cdots \times S_m\lo R$ be two ring
isomorphisms, in which $R_i$'s and $S_j$'s are indecomposable isoartinian rings and $l, m \in \Bbb N$. For each $1\leq i\leq l$
and $1\leq j\leq m$, set $A_i=\phi(0\times \cdots \times 0\times R_i\times 0\times \cdots \times 0)$ and $B_j=\psi(0\times \cdots
\times 0\times S_j\times 0\times \cdots \times 0)$. Then $$R=A_1\oplus\cdots\oplus A_l=B_1\oplus\cdots \oplus B_m$$ and each $A_i$
as well as each $B_j$ is indecomposable as an ideal, because $R_i$'s and $S_j$'s are indecomposable as rings. Now, by
\cite[Lemma 3.8]{LA}, $l=m$ and $A_i=B_{\sigma (i)}$ for some permutation $\sigma$ of the set $\{1,\dots, l\}$. It follows that
$R_i\cong S_{\sigma (i)}$ for every $i=1,\dots, l$.
\end{proof}	

Let $\Nil R$ denote the nilradical of $R$.

\begin{lemma}\label{2.7g} Let $R$ be an isoartinian ring. Then the following statements hold.
\begin{enumerate}
\item[(i)] $\Nil R$ is nilpotent.
\item[(ii)] Every two distinct minimal prime ideals of $R$ are coprime.
\end{enumerate}
\end{lemma}

\begin{proof} (i) follows by \cite[Proposition 2.3]{Da}.

(ii) Let $\fp$ and $\fq$ be two distinct minimal prime ideals of $R$, and let $A:=R/\fp \times R/\fq$. By Lemma \ref{2.2b}, $\fp$ and
$\fq$ belong to $\Ass R$, and so $R/\fp$ and $R/\fq$ are principal ideal domains by Proposition \ref{2.3c}. Since $A$ is the product
of two principal ideal domains, we can conclude that every ideal of $A$ is isomorphic to a direct summand of $A$. Thus, by \cite[Theorem 2.3]{BDV1}, there exist prime ideals $P_1, \dots, P_k$ of $A$ such that $A\cong A/P_1\times \dots \times A/P_k$, and $P_i$'s are comparable
or coprime. By applying Lemma \ref{2.6f} and \cite[Lemma 4.1]{FN3}, we can conclude that $k=2$, $A/\fp\cong A/P_{\sigma (1)}$, and
$A/\fq\cong A/P_{\sigma (2)}$ for some permutation $\sigma$ of the set $\{1,2\}$. As every ring homomorphism between two quotient rings
of $A$ is an $A$-homomorphism, we conclude that $$\fp=\Ann_A(A/\fp)=\Ann_A(A/P_{\sigma (1)})=P_{\sigma (1)}.$$ Similarly, we have
$\fq=P_{\sigma (2)}$. Since $\fp$ and $\fq$ are two distinct minimal primes, they are not comparable and therefore coprime.
\end{proof}

The next result provides a simple criterion for an isoartinian ring to be indecomposable.

\begin{lemma}\label{2.8h} An isoartinian ring $R$ is indecomposable if and only if $\Min R$ is a singleton.
\end{lemma}

\begin{proof} First, assume that $R$ is decomposable. Then, there is a natural number $t\geq 2$ and rings $R_1, \dots, R_t$
such that $R\cong R_1\times \dots \times R_t$. This, in particular, implies that $\Min R$ has at least $t$ elements.

Next, assume that $\Min R=\{\fp_1, \dots, \fp_t\}$, with $t\geq 2$. Note that by Lemma \ref{2.2b}, $\Min R$ is finite. By
Lemma \ref{2.7g}(ii), every two distinct elements $\fp$ and $\fq$ of $\Min R$ are coprime, and so $\fp^k+\fq^k=R$ for every
natural number $k$. By Lemma \ref{2.7g}(i), $\Nil R$ is nilpotent, and hence $\fp_1^k \cdots \fp_t^k=0$ for some natural
number $k$. Now, the Chinese remainder theorem \cite[Proposition 1.10]{AM} yields
$$\begin{array}{ll}
R&\cong R/0\\
&\cong R/(\fp_1^k \cdots \fp_t^k)\\
&\cong R/(\fp_1^k \cap \cdots \frak \cap p_t^k)\\
&\cong (R/\fp_1^k)\times \cdots \times (R/\fp_t^k),
\end{array}$$
and so $R$ is decomposable.
\end{proof}

Although any local ring is indecomposable, the indecomposable isoartinian ring $\mathbb{Z}$ shows that $\Min R$ cannot be replaced by $\Max R$ in the above lemma.

Using Proposition \ref{2.3c} and Lemma \ref{2.8h}, we can completely determine the spectrum of an isoartinian local ring.

\begin{corollary}\label{2.9i} Let $(R,\fm)$ be an isoartinian local ring. Then $\Spec R=\{\fp, \fm\}$, where $\fp$ is the unique
minimal prime ideal of $R$. Moreover if $\fm$ contains a regular element, then $\fp=\bigcap_{i=1}^\infty {\fm}^i$ and $\fp\neq
\fm$.
\end{corollary}

\begin{proof} Since $R$ is local, it is indecomposable. Hence, Lemma \ref{2.8h} yields that $R$ has a unique minimal prime
ideal $\fp$. Proposition \ref{2.3c} implies that $\Spec R=\Min R\cup \Max R=\{\fp, \fm\}$.

Now, assume that $\fm$ contains a regular element. By Lemma \ref{2.4d}, $\fm$ is principal and is generated by a regular element
$z$. For every natural number $i$, we show that there is no ideal between ${\fm}^i$ and ${\fm}^{i-1}$. Clearly, ${\fm}/{\fm}^i$
is the unique prime ideal of the ring $R/{\fm}^i$, and ${\fm}/{\fm}^i$ is cyclic. Hence, by Cohen's theorem, $R/{\fm}^i$ is a
Noetherian ring. As $\dim R/{\fm}^i=0$, by \cite[Theorem 8.5]{AM}, we deduce that $R/{\fm}^i$ is Artinian. Hence, by
\cite[Proposition 8.8]{AM}, $R/{\fm}^i$ is a special ring, and so $R/\fm^i, \fm/\fm^i, \dots, \fm^{i-1}/\fm^i, 0$ are the only
ideals of $R/\fm^i$. Thus, there is no ideal strictly between ${\fm}^i$ and ${\fm}^{i-1}$.
	
Set ${\fr}=\bigcap_{i=1}^\infty {\fm}^i$. We show that ${\fr}$ is prime. Let $a$ and $b$ be two elements of $R\setminus \fr$.
We show that $ab\not \in \fr$. There are natural numbers $l$ and $t$ such that $a\in {\fm}^l\setminus {\fm}^{l+1}$ and
$b\in {\fm}^t\setminus {\fm}^{t+1}$. Thus, by the above argument, ${\fm}^l=Ra+{\fm}^{l+1}$ and ${\fm}^t=Rb+{\fm}^{t+1}$. So, $${\fm}^{l+t}={\fm}^l{\fm}^t=(Ra+{\fm}^{l+1}) (Rb+{\fm}^{t+1})\subseteq Rab+{\fm}^{l+t+1}.$$ Suppose that $ab\in {\fm}^{l+t+1}$.
Then ${\fm}^{l+t}={\fm}^{l+t+1}$, and so $z^{l+t}=rz^{l+t+1}$ for some $r\in R$. Since $z$ is regular, we get $rz=1$, which is
a contradiction. Consequently $ab\notin {\fm}^{l+t+1}$, and so $ab\not \in \fr$. If $\fr={\fm}$, then  $z\in \fm^2$, and so $z=rz^2$
for some $r\in R$. As $z$ is regular, we deduce that $z$ is a unit, a contradiction. Therefore $\fr=\fp$.
\end{proof}

Next, we state the following easy observation.

\begin{lemma}\label{2.10j} Let $\fp$ and ${\fq}=Rz$ be two prime ideals of a ring $R$ with ${\fp}\subsetneq \fq$. Then $\fp\fq={\fp}
z=\fp$.
\end{lemma}

\begin{proof} Clearly, $\fp\fq={\fp} z\subseteq \fp$. If $x\in {\fp}$, then $x=rz$ for some $r\in R$. Since $z\not \in
{\fp}$, we have $r\in {\fp}$. Hence $x=rz\in {\fp} z$, and the proof is complete.
\end{proof}	

Finally, we are ready to prove Theorem \ref{1.1}. This theorem is  analogous to the classical decomposition theorem for Artinian
rings, which was proved by Akizuki and Cohen.

\vspace{.4cm}

{\bf Proof of Theorem 1.1:} Based on Lemma \ref{2.6f} and \cite[Lemma 4.1]{FN3}, we can conclude that $R$ is isoartinian if and only
if it is uniquely isomorphic to the direct product of finitely many indecomposable isoartinian rings. Hence, to complete the proof,
it remains to establish that every indecomposable isoartinian ring $A$ takes on one of the six given forms.

Let $A$ be an indecomposable isoartinian ring. We note that $\dim A\leq 1$ by Proposition \ref{2.3c}, and as $A$ is indecomposable,
by Lemma \ref{2.8h}, $\Min A$ has a unique element $\fp$. Hence $\Nil A=\fp$, and so $\fp$ is nilpotent by Lemma \ref{2.7g}(i). The
proof is broken into two cases.

{\bf Case 1.} $\dim A=0$.

Then $\Spec A=\{\fp\}$. If $A$ is Noetherian, then $A$ is Artinian (type (i)). If $A$ is not Noetherian, then type (ii) occurs.

{\bf Case 2.} $\dim A=1$.

Lemma \ref{2.2b} implies that $\fp\in \Ass A$, and so $A/\fp$ is a principal ideal domain by Proposition \ref{2.3c}. As $\dim A=1$,
it follows that $\fp$ is not maximal, and so $A/\fp$ is not a field. If $\fp=0$, then $A$ is a non-field principal ideal domain
(type (iii)).

In the rest of the proof, we may and do assume that $\fp\neq 0$. Assume that $A$ is Noetherian. We claim that $\Min A\varsubsetneq
\Ass A$. Suppose on the contrary $\Min A=\Ass A$. Let $\fm$ be a maximal ideal of $A$. Then $\fp\subsetneq \fm$, and so $\fm$ contains
a regular element. Now, Lemma \ref{2.4d} implies that $\fm=Az$ for some regular element $z$. By Lemma \ref{2.10j}, it follows that
$\fp=\fm \fp$. By localizing at $\fm$ and applying Nakayama's lemma, it follows that $\fp A_{\fm}=0$. Thus $\fp A_{\fn}=0$ for every
maximal ideal $\fn$ of $A$, and so $\fp=0$. We have arrived at a contradiction, and so type (iv) occurs.

Next, suppose that $A$ is not Noetherian. Then either type (v) or type (vi) occurs.   \hspace{1.5cm} $\square$

\section{Examples}

In this section, we present examples that demonstrate the occurrence of types (ii), (iv), (v), and (vi) stated in Theorem \ref{1.1}.
These examples are provided in Examples \ref{3.1a}, \ref{3.3c}, \ref{3.4d}, and \ref{3.7g}. Furthermore, we illustrate
with some examples that the condition ``isoartinian" in types (ii), (v) and (vi) of Theorem \ref{1.1} cannot be relaxed; see Examples
\ref{3.2b}, \ref{3.5e} and \ref{3.8h}.

Specifically, our first example demonstrates the occurrence of type (ii) in Theorem \ref{1.1} and provides an instance of an
isoartinian ring that is not isonoetherian.

\begin{example}\label{3.1a} Let $F$ be a field and $R=F[x_1,x_2,x_3,\dots]$, where $x_ix_j=0$ for each $i,j\in \Bbb N$. Clearly, $R$
is a local ring with the unique prime ideal ${\fm}=\langle x_1, x_2, \dots \rangle$. Since $\fm^2=0$, we can deduce that $R$ can be
expressed as $R=F\oplus(\oplus_{i\in \mathbb{N}} Fx_i)$. Thus, $R$ is an $F$-vector space with a countable basis and any ideal $I$
of $R$ can be written as $F^{(\Gamma)}$, where $\Gamma$ is a countable set.

Now, consider a descending chain of ideals of $R$: $$I_1\supseteq I_2\supseteq \cdots \supseteq I_n \supseteq \cdots .$$ If there
exists an $I_j$ isomorphic to $F^{(\Gamma)}$ with $\Gamma$ finite, then the chain is stationary. If there is no such ideal $I_j$,
then we must have $I_1\cong I_2\cong \cdots$. This shows that $R$ is an isoartinian ring.

However, $R$ is not isonoetherian. To see this, for each natural number $k$, let $J_k:=\langle x_1, x_2, \dots, x_k\rangle$. Then $J_k$ is an $F$-vector space of dimension $k$, and the chain $$J_1\subseteq J_2\subseteq \cdots \subseteq J_k\subseteq \cdots$$ does not stabilize up to isomorphism. Hence, $R$ is not isonoetherian and, consequently, not Noetherian.
\end{example}

We now present an example of a non-isoartinian local ring with a nilpotent maximal ideal.

\begin{example}\label{3.2b} Let $F$ be a field and $R=F[x_1,x_2,x_3,\dots]$, where $x_ix_j=0$ and $x_i^3=0$ for all natural numbers $i<j$. Clearly, $R$ is a non-Noetherian local ring with the unique prime ideal ${\fm}=\langle x_1, x_2, \dots \rangle$ and $\fm^3=0$. For each
natural number $k$, set $J_k:=\langle x_{k}, x_{k+1}, \dots \rangle$. These ideals form the following descending chain: $$J_1\supsetneq J_2\supsetneq \cdots \supsetneq J_k\supsetneq \cdots .$$ It is easy to see that $x_k\in (0:_RJ_{k+1})\setminus (0:_RJ_{k})$. Since isomorphic modules have the same annihilator, this chain does not stabilize up to isomorphism.
Thus, $R$ satisfies all conditions in type (ii) of Theorem \ref{1.1} except the isoartinian condition.
\end{example}

Let $J(R)$ denote the Jacobson radical of $R$. Recall that a ring $R$ is said to be {\it perfect} if $R/J(R)$ is a semisimple ring
and $J(R)$ is $t$-nilpotent. A ring $R$ is called {\it subperfect} if its total ring of fractions is a perfect ring. An $R$-module
$M$ is called isoartinian if every descending chain of submodules of $M$ terminates up to isomorphism.

Next, we provide an example of rings of type (iv) in Theorem \ref{1.1}. It is well-known that every Artinian ring is perfect. This raises
the question of whether every isoartinian ring is also perfect. However, the following example shows that an isoartinian ring may not even
be subperfect. Furthermore, as every Artinian ring is Cohen-Macaulay, one may conjecture that any Noetherian isoartinian ring is also Cohen-Macaulay. Nevertheless, the next example demonstrates that this is not true.

\begin{example}\label{3.3c} Let $F$ be a field and $R=F[[X,Y]]/\langle XY,Y^2 \rangle$. Let $x$ and $y$ denote the residue classes of $X$
and $Y$ in $R$, and set $\fm=\langle x,y \rangle$ and $\fp=Ry$. As $0=\langle x,y^2 \rangle \cap \fp$ is a minimal primary decomposition
of the zero ideal of $R$, it follows that $\Ass R=\{\fp,\fm\}$, and so $\Min R\subsetneq \Ass R$. It is easy to verify that $(0:_Ry)=\fm$,
and so $\fp$ is a simple $R$-module. We prove the following claim:

{\bf Claim.} Let $J$ be an ideal of $R$ such that $J\cap \fp=0$. Then $J$ is an isoartinian $R$-module.

{\bf Proof.} As $R/\fp$ is a principal ideal domain and $(J+\fp)/\fp$ is an ideal of $R/\fp$, it follows that $(J+\fp)/\fp$ is an
isoartinian $R/\fp$-module. Thus, $(J+\fp)/\fp$ is also isoartinian as an $R$-module. Note that the set of $R$-submodules of
$(J+\fp)/\fp$ coincides with the set of its $R/\fp$-submodules, and every $R/\fp$-homomorphism between two ideals of $R/\fp$ is also
an $R$-homomorphism. Since $J\cong J/J\cap \fp\cong (J+\fp)/\fp$, we deduce that $J$ is an isoartinian $R$-module.

Next, we show that $R$ is isoartinian. Let $$I_1\supseteq I_2\supseteq \cdots \supseteq I_n \supseteq  \cdots$$ be a descending chain
of ideals of $R$. Assume that $I_{\ell}\cap \fp=0$ for some natural number $\ell$. Then $I_{\ell}$ is an isoartinian $R$-module by
the above claim. This implies that the chain must be stationary up to isomorphism. Now, assume that $I_{\ell}\cap \fp=\fp$ for all
natural numbers $\ell$. Since $\fm=\fp\oplus Rx$, we can easily see that $I_{\ell}=\fp \oplus (I_{\ell}\cap Rx)$. By the above claim,
$I_{\ell}\cap Rx$ is an isoartinian $R$-module for all $\ell\in \mathbb{N}$. As $I_1\cap Rx$ is an isoartinian $R$-module, there is a
natural number $k$ and $R$-isomorphisms $\phi_i:I_{k}\cap Rx\lo I_{k+i}\cap Rx$ for all $i\geq 1$. In view of the expression for
$I_{\ell}$, $\phi_i$ can be extended to an $R$-isomorphism $\psi_i:I_{k}\lo I_{k+i}$ for all $i\geq 1$. Hence, the chain must be
stationary up to isomorphism, which implies that $R$ is isoartinian.

Since $\Min R\subsetneq \Ass R$, \cite[Lemma 2.2]{FO} implies that $R$ is not subperfect. Additionally, since $\Min R\subsetneq \Ass R$,
we deduce that $R$ is not Cohen-Macaulay.
\end{example}

We follow by giving an example of rings of type (v) in Theorem \ref{1.1}.

\begin{example}\label{3.4d} Let $F$ be a field and $R$ be the quotient of the algebra $F[[x_1,x_2,\dots]]$ modulo the relations:
$$\begin{cases}
x_1x_i=0 \  \  \  \ \text{for all}\ i\geq 2\\
x_ix_j=0 \ \  \ \ \text{for all}\  i\geq 2\ \text{and}\ j\geq 2.
\end{cases}
$$
Set $\fp=\langle x_2, x_3, \dots \rangle$ and $\fm=\langle x_1, x_2, \dots \rangle$. Clearly, $\fp^2=0$. It follows that $\fp$ is
the unique minimal prime ideal of $R$, and $\fm$ is the unique maximal ideal of $R$. It is easy to check that $\fm=Rx_1\oplus \fp$,
$(0:_Rx_1)=\fp$, $(0:_Rx_2)=\fm$, and $\text{Soc}(R)=\fp$. Hence, $\Min R\subsetneq \Ass R$. We have $\fm=Rx_1\oplus(\oplus_{i>1}Fx_i)$,
and $Rx_1\cong R/\fp$ is a principal ideal domain. Thus, by the equivalence of the conditions 3(ii) and 4(iii) of \cite[Theorem 4.5]{BDV1},
there exists $0\neq x\in R$ such that every proper ideal of $R$ is semisimple or is isomorphic to $Rx\oplus J$, where $J\subseteq
\text{Soc}(R)$. Every semisimple ideal of $R$ is contained in $\text{Soc}(R)$, which is an $F$-vector space with a countable basis.
Thus, by an argument similar to that given in Example \ref{3.1a}, we can show that $R$ is isoartinian. Clearly, $R$ is not Noetherian.
\end{example}

In the following example, we demonstrate that the condition of being ``isoartinian" in type (v) of Theorem \ref{1.1} cannot be
relaxed.

\begin{example}\label{3.5e} Let $F$ be a field and $R$ be the quotient of the algebra $F[[x,y_1,y_2,\dots]]$ modulo the relations:
$$\begin{cases}
xy_i=0 \  \  \  \ \text{for all}\ i\geq 1\\
y_iy_j=0 \ \  \ \text{for all}\ j>i \\
y_i^3=0\ \  \ \ \  \text{for all}\ i\geq 2\\
y_1^2=0.
\end{cases}
$$
Set $\fp=\langle y_1,y_2,\dots\rangle$ and $\fm=\langle x,y_1,y_2,\dots\rangle$. As $\fp^3=0$, it follows that $\fp$ is the unique
minimal prime ideal of $R$. Since $R/\fp\cong \Bbb F[[x]]$ is a DVR, we deduce that $\dim R=1$. It is easy to check that $\fp=(0:_Rx)$
and $\fm=(0:_Ry_1)$, and so $\Min R\subsetneq \Ass R$. Clearly, $R$ is not Noetherian.

We will now show that $R$ is not isoartinian. To do this, assume that $J_k=\langle y_{\scalebox{0.5}{2k}},y_{\scalebox{0.5}{2k+1}},\dots
\rangle$, and consider the following descending chain of ideals in $R$:
$$J_1\supsetneq  J_2\supsetneq  J_3\supsetneq \cdots .$$
Suppose $J_{k}\cong J_{k+1}$ for some $k\in \Bbb N$. It is clear that $y_{\scalebox{0.5}{2k+1}}\in (0:_RJ_{k+1})$, but $y_{\scalebox{0.5}{2k+1}}\not\in (0:_RJ_{k})$. This contradicts the fact that isomorphic modules have the same annihilator. Thus, $R$
satisfies all conditions in type (v) of Theorem \ref{1.1} except for being isoartinian.
\end{example}

Our next example shows that type (vi) in Theorem \ref{1.1} can occur. To present it, we need the following lemma.

We recall that a ring $R$ is a {\it chain ring} if the set of ideals of $R$ is totally ordered with respect to inclusion
(this is equivalent to the condition that the set of principal ideals of $R$ is totally ordered with respect to inclusion).
It is well-known that a chain ring is local. A chain domain is called a {\it valuation domain}.

\begin{lemma}\label{3.6f} Let $A$ be a DVR with the maximal ideal $Az$. Let $K$ be the field of fractions of $A$, and set
$R=A+xK[[x]]$, $\fp=xK[[x]]$, and $\fm=Az+\fp$. Then
\begin{enumerate}
\item[(i)] $R$ is a valuation domain.
\item[(ii)] $R$ is a non-Noetherian domain, $\fm=Rz$, and $\dim R=2$.
\item[(iii)] $\Spec R=\{0,\fp,\fm\}$.
\item[(iv)] $\bigcap_{i=1}^\infty {\fm}^i=\fp$.
\end{enumerate}
\end{lemma}

\begin{proof} (i) Since $R$ is a subring of $K[[x]]$, it follows that $R$ is a domain. Let $f$ and $g$ be two nonzero elements
of $R$. As $K[[x]]$ is a DVR, without loss of generality, we may assume that $g|f$ in $K[[x]]$, that is $f=hg$ for some
$h\in K[[x]]$. We want to show that $f|g$ or $g|f$ in $R$. Let $h=h_0+h_1x+h_2x^2+\cdots$. If $h_0=0$, then
$$h=x\left(h_1+h_2x+\cdots \right)\in xK[[x]]\subseteq R,$$ and so $g|f$ in $R$, and we are done. Next, assume that $h_0\neq 0$.
Since $A$ is a valuation domain, we have either $h_0\in A$ or $h^{-1}_0\in A$. We break the proof into two cases:

\noindent \textbf{Case 1}. $h_0\in A$. Then $h=h_0+x\left(h_1+h_2x+\cdots \right)\in R$, and so $g|f$ in $A$, and we are done
again.

\noindent \textbf{Case 2}. $h^{-1}_0\in A$. As $h_0$ is nonzero, $h$ is an invertible element in $K[[x]]$. Clearly, $h^{-1}$ has the form $$h^{-1}=h^{-1}_0+l_1x+l_2x^2+\dots =h^{-1}_0+x(l_1+l_2x+\cdots )\in R.$$ So, from $h^{-1}f=g$, we deduce that $f|g$ in $R$. This means that $R$ is a valuation domain.

(ii) By \cite[Corollary 14(a)]{DK}, we conclude that $\dim R=2$. By \cite[Proposition 6]{DK}, $\fm$ is the unique maximal ideal of $R$.
Since $R$ is a valuation domain and $z\notin \fp$, it follows that $\fp\subseteq Rz$. Hence, $\fm=Az+\fp\subseteq Rz$, and so $\fm=Rz$.
If $R$ were Noetherian, then Krull's principal ideal theorem would imply that $\dim R=\Ht \fm\leq 1$. So, $R$ is non-Noetherian.

(iii) Obviously, $\fp$ is a prime ideal of $R$. Suppose that $R$ possesses a prime ideal $\fq$ other than the prime ideals $0$, $\fp$,
and $\fm$. As $R$ is a valuation domain, either $\fp\subsetneq \fq$ or $\fq\subsetneq \fp$. Thus, we can conclude that either the chain $0\subsetneq \fp\subsetneq \fq\subsetneq \fm$ or the chain $0\subsetneq \fq\subsetneq \fp\subsetneq \fm$ exists, implying that
$\dim R\geq 3$. Since, by (ii), $\dim R=2$, we deduce that $\Spec R=\{0,\fp,\fm\}$.

(iv) For a proper ideal $I$ of a chain ring $T$, by \cite[Lemma 1.3, Chapter II]{LS}, either some power of $I$ is zero or
$\bigcap_{i=1}^\infty I^i$ is a prime ideal of $T$. Since $R$ is a domain and $\fm\neq 0$, no power of $\fm$ is zero, and so
$\bigcap_{i=1}^\infty {\fm}^i$ is a prime ideal of $R$. As $R$ is non-Noetherian, Lemma \ref{2.5e} asserts that
$\bigcap_{i=1}^\infty {\fm}^i\neq 0$. If $\bigcap_{i=1}^\infty {\fm}^i=\fm$, then $\fm=\fm^2$, which implies that $\fm=0$ by
Nakayama's lemma. Thus, $\bigcap_{i=1}^\infty {\fm}^i=\fp$.
\end{proof}

\begin{example}\label{3.7g} Let $A$, $K$, $R$, $\fp$, and $\fm=Rz$ be as in the lemma above. By Lemma \ref{3.6f}, $R$
is a non-Noetherian valuation domain of dimension two, and $\bigcap_{i=1}^\infty {\fm}^i=\fp$. Set $I=Rx^2$, $D=R/I$
and $\bar{\fp}=\fp/I$. Since every factor of a chain ring is again a chain ring, the ring $D$ is a chain ring as well.
As $\Spec R=\{0,\fp,\fm\}$, it follows that $\Spec D=\{\bar{\fp},\fm/I\}$. Lemma \ref{2.10j} yields that $\fm \fp=\fp$.
Since $\fp$ is nonzero, by Nakayama's lemma, it follows that $\fp$ is not finitely generated. This implies that the ideal
$\bar{\fp}$ is also not finitely generated.

We claim that $D$ is an isoartinian ring. Since $z+I$ is a regular element of $D$, we have that $\fm/I$ is not an
associated prime ideal of $D$. On the other hand, we can verify that $\bar{\fp}=(0:_D(x+I))$, and so $\bar{\fp}\in \Ass D$.
Hence, $\Min D=\Ass D$. Suppose that we have a decreasing chain of ideals $$I_1\supseteq I_2\supseteq I_3\supseteq \cdots$$ in
$D$. Assume that $\bar{\fp}\subsetneq I_j$ for every $j\in \mathbb{N}$. Since $R/\fp\cong A$ is a DVR, it follows that
$\{0\}\cup \{\fm^i/\fp \mid i\in \mathbb{N}_0\}$ is the set of all ideals of $R/\fp$. It can be easily checked that each
$I_j$ is a power of the maximal ideal $\fm/I$, and hence they are principal and generated by a power of the regular element
$z+I$. Therefore, they are isomorphic to $D$, and the chain stops up to isomorphism.

Now, assume that there is a natural number $j_1$ such that $I_{j_1}\subseteq \bar{\fp}$. If $I_j=\bar{\fp}$ for all
$j\geq j_1$, then we are done. So, without loss of generality, we may assume that $I_{j_1}\subsetneq \bar{\fp}$. Let
$\bar{y}\in \bar{\fp}\setminus I_{j_1}$. Then $I_{j_1}\subsetneq D\bar{y}\subsetneq \bar{\fp}$, because $D$ is a chain
ring. As $\bar{\fp}\bar{y}=0$, we see that $D\bar{y}$ is a well-defined $D/\bar{\fp}$-module. Note that
$D/\bar{\fp}\cong R/\fp$ is a DVR. Thus, by the structure theorem for finitely generated modules over a principal ideal
domain, every decreasing chain of submodules of the $D/\bar{\fp}$-module $D\bar{y}$ must stop up to isomorphism after a
finite number of steps. This forces the chain to stop up to isomorphism. Note that $D$-submodules of $D\bar{y}$ coincide
with its $D/\bar{\fp}$-submodules.
\end{example}

Our last example below shows that the condition of being ``isoartinian" in type (vi) of Theorem \ref{1.1} cannot be relaxed.

\begin{example}\label{3.8h} Let $R$ be the quotient of the algebra $\Bbb Z[x_1,x_2,\dots]$ modulo the relations:
$$\begin{cases}
x_ix_j=0 \ \  \ \text{for all}\ j>i \\
x_i^3=0\ \  \ \ \ \  \text{for all}\ i\geq 1.
\end{cases}
$$
Set $\fp=\langle x_1,x_2,\dots \rangle$. As $\fp^3=0$, it follows that $\fp$ is the unique minimal prime ideal of $R$. It is easy
to see that $\fp=(0:_Rx_1^2)$, so $\fp\in \Ass R$. On the other hand, it is straightforward to check that the zero ideal of $R$ is
$\fp$-primary. Thus, $Z(R)=\fp$, which implies that $\Ass R=\{\fp\}=\Min R.$ Clearly, $R$ is not Noetherian. Let $J_k=\langle x_{\scalebox{0.5}{2k}}, x_{\scalebox{0.5}{2k+1}},\dots \rangle$ for all $k\in \mathbb{N}$. With the same argument as in Example
\ref{3.5e}, we see that the following descending chain of ideals of $R$ does not terminate up to isomorphism:
$$J_1\supsetneq  J_2\supsetneq  J_3\supsetneq \cdots .$$
Hence, $R$ satisfies all conditions in type (vi) of Theorem \ref{1.1} except for being isoartinian.
\end{example}

We have not provided an example to show that the condition of being ``isoartinian" in type (iv) of Theorem \ref{1.1} cannot be relaxed.
In fact, we conclude the paper by proposing the following question:

\begin{question}\label{3.9i} Let $R$ be a one-dimensional Noetherian ring with a unique minimal prime ideal $\fp$ such that $R/\fp$
is a principal ideal domain and $\Min R\subsetneq \Ass R$. Is $R$ necessarily isoartinian?
\end{question}

\section*{Acknowledgments}

We would like to thank the anonymous referees for their comments, which helped improve the paper's presentation.

\end{document}